\documentclass[12pt]{amsart}
\usepackage{amsmath,amssymb,fullpage}
\usepackage{enumerate}
\usepackage{epstopdf}
\usepackage{subfigure}
\usepackage{color}
\usepackage{xcolor}
\usepackage{float}
\usepackage{colortbl}
\usepackage{tikz}
\usepackage{amsmath}
\usepackage[all]{xy}

\theoremstyle{plain}
\newtheorem{theorem}{Theorem}[section]
\newtheorem{corollary}[theorem]{Corollary}
\newtheorem{lemma}[theorem]{Lemma}
\newtheorem{proposition}[theorem]{Proposition}
\newtheorem{definition}[theorem]{Definition}
\newtheorem{example}[theorem]{Example}

\theoremstyle{remark}
\newtheorem{remark}{Remark}

\def\bal{\begin{array}{ll}}
\def\eal{\end{array}}

\numberwithin{equation}{section}

\title{Partitions of the polytope of Doubly Substochastic Matrices }

\author{Lei Cao}
\address{School of Mathematics and Statistics, Shandong Normal University, Shandong, 250358, China;\\
Department of Mathematics, Georgian Court University, NJ, 08701, USA}
\email{lcao@georgian.edu}

\author{Zhi Chen$^\ast$}
\address{Department of Mathematics, Nanjing Agricultural University, Jiangsu, 210095, China}
\email{chenzhi@njau.edu.cn }

\thanks{ $^\ast$ Corresponding author. Email: chenzhi@njau.edu.cn}


\begin{document}
\maketitle

\begin{abstract}
In this paper, we provide three different ways to partition the polytope of doubly substochastic matrices into subpolytopes via the prescribed row and column sums, the sum of all elements and the sub-defect respectively. Then we characterize the extreme points of each type of convex subpolytopes. The relations of the extreme points of the subpolytopes in the three partitions are also given.
\end{abstract}

\section{Introduction}
\begin{definition}
An $n\times n$ matrix $A=[a_{i,j}]_{n\times n}$ is called a {\sl doubly stochastic matrix} if it satisfies
\begin{enumerate}[(a)]
\item $a_{i,j}\geq 0$, and
\item \label{item2}$\sum_{s=1}^n a_{i,s}=1, \sum_{t=1}^n a_{t,j}=1$,
\end{enumerate}
for all $1\leq i, j\leq n.$
\end{definition}
The definition of {\it doubly substochatic matrices} can be obtained by replacing the equalities in (\ref{item2}) by the inequalities.
\begin{definition}
An $n\times n$  matrix $B=[b_{i,j}]_{n\times n}$ is called a {\sl doubly  substochastic matrix} if it satisfies
\begin{enumerate}[(a)]
\item $b_{i,j}\geq 0$, and
\item \label{item2}$\sum_{s=1}^n b_{i,s}\leq 1, \sum_{t=1}^n b_{t,j}\leq 1$,
\end{enumerate}
for all $1\leq i, j\leq n.$
\end{definition}
Denote the set of all $n\times n$ doubly stochastic matrices by $\Omega_n$, and the set of all $n\times n$ doubly substochastic matrices by $\omega_n$. The set $\Omega_n$ is a convex polytope, and has been intensively studied by many mathematicians~\cite{GPE1980,DIF1981,PMG1966,MK1970,MK1972,MMHM1962,MMHM1964,VDW,Gen2,Gen1,sinkhorn1967,MO,CR}. Specially, the extreme points of $\Omega_n$ are exactly the permutation matrices due to Birkhoff~\cite{GB1946} and von Neumann \cite{VNJ}, which can be stated as follows.
\begin{theorem}{\rm \cite{GB1946,VNJ}}\label{thm1}
An $n\times n$ matrix $A$ is a doubly stochastic matrix if and only if there are finite permutation matrices $P_1,P_2, \cdots, P_N$ and positive numbers $\alpha_1,\cdots,\alpha_N$ such that $\alpha_1+\cdots+\alpha_N=1$ and $A=\alpha_1 P_1+\cdots+\alpha_N P_N.$
\end{theorem}
The set $\omega_n$ is also a convex polytope and its extreme points are partial permutation matrices~\cite{MI1959}, i.e., matrices with at most one element in each row and each column equal to one and other elements zero. Since $\Omega_n \subseteq \omega_n$, one may wonder if the classical results of doubly stochastic matrices can be extended to doubly substochastic matrices. However, the question becomes meaningless sometimes, because the results are either the same or trivial even though $\omega_n$ is a much bigger polytope. For example, the maximum diagonal or the upper bound of Frobenius norm would be the same no matter either $\Omega_n$ or $\omega_n$ is considered, while the minimal diagonal of $\omega_n$ is simply zero since zero matrix is contained in $\omega_n$. The reason is that it is too coarse to consider those characteristics on the entire polytope of doubly substochastic matrices.
Therefore we try to divide $\omega_n$ into subsets on which those characteristics become more meaningful.

We consider three different ways to partition $\omega_n,$ through which one may extend the classical results of $\Omega_n$ to $\omega_n$. Some notations will be introduced in the next section. This paper is organized as follows. In Section~\ref{sec0}, we introduce three partitions of $\omega_n$ and three types of subpolytopes $\omega_n(R,S), \omega_n^s$ and $\omega_{n,k}$, which are obtained from these three partitions respectively. In Section~\ref{sec1}, the extreme points of the $\omega_n(R,S)$ are characterized via the connection between $\omega_n(R,S)$ and transportation polytopes. In Section~\ref{sec2} and Section~\ref{sec3}, we characterize the extreme points of $\omega_n^s$ and show the results in two different approaches. Via the extreme points of $\omega_n^s$,  we characterize the extreme points of the subpolytope $\omega_{n,k}$ in Section~\ref{sec4}. In the end we summarize some results of doubly substochastic matrices as extensions of the results of doubly stochastic matrices in Section~\ref{sec5}.

\section{Three different partitions of $\omega_n$}\label{sec0}
The first way to partition $\omega_n$ is induced by a characteristic of doubly substochastic matrices called sub-defect~\cite{CKP2016,CK2017, LeiGLA}. \begin{definition}{\rm (\cite{CKP2016})}
The {\sl sub-defect} of an $n\times n$ doubly substochastic matrix $B$, denoted by $sd(B),$ is defined to be the smallest integer $k$ such that there exists an $(n + k) \times (n + k)$ doubly stochastic matrix containing $B$ as a submatrix.
\end{definition}
Denote the sum of all elements of a matrix $B=[b_{i,j}]$ by $\sigma(B)$, i.e.,
\begin{equation}\label{SDE}
\sigma(B) =\sum_{i=1}\sum_{j=1} b_{i,j}.
\end{equation}
It has been shown that the sub-defect $k$ can be calculated easily by taking the ceiling of the difference of its size and the sum of all elements.
\begin{theorem}{\rm (Theorem 2.1, \cite{CKP2016})}\label{thm2.1}
Let $B = [b_{i,j}]$ be an $n \times n$ doubly substochastic
matrix. Then $$sd(B) =\lceil n-\sigma(B) \rceil.$$   where $\lceil x \rceil$ is the ceiling of $x.$
\end{theorem}
Denote by $\omega_{n,k}$ the set of all $n \times n$ doubly substochastic matrices with sub-defect $k,$ i.e.,
\begin{equation} \label{p3}
\omega_{n,k}=\{B\in\omega_n \ | \ sd(B)=k\}.
\end{equation}
Let \begin{equation}\label{part1}\mathcal{P}_1=\{\omega_{n,0}, \omega_{n,1},\ldots, \omega_{n,n}\}.\end{equation} Then $\mathcal{P}_1$ is a partition of $\omega_n$ since
\begin{enumerate}
\item $\emptyset \notin \mathcal{P}_1;$
\item $\omega_{n,i}\cap \omega_{n,j}=\emptyset$ for $i\neq j;$
\item $\bigcup_{i=0}^n\omega_{n,i}=\omega_n.$
\end{enumerate}
The second way to partition $\omega_n$ is using the sum of all elements. Let $0\leq s\leq n$, denote by $\omega_n^s$ the set of all matrices in $\omega_n$ such that the sum of all elements equals to $s$,  i.e.,
\begin{equation} \label{p2}
\omega_n^s=\{B\in\omega_n \ | \ \sigma(B)=s\}.
\end{equation}
Let \begin{equation}\label{part2} \mathcal{P}_2= \{ \omega_n^s \ |\  0\leq s\leq n \}. \end{equation} Then $\mathcal{P}_2$ is a partition of $\omega_n,$ because
\begin{enumerate}
\item $\emptyset\notin \mathcal{P}_2;$
\item $\omega_n^s \cap \omega_n^t =\emptyset$, if $s\neq t;$
\item $\displaystyle \bigcup_{0\leq s\leq n}\omega_n^s=\omega_n$.
\end{enumerate}
The third way to partition $\omega_n$ is triggered by a special case of the {\sl transportation polytopes}~\cite{EB1972}, in the case when they are square matrices. Transportation polytopes are formed by all nonnegative matrices (not necessarily square) with prescribed row sums and column sums. For an $n\times n$ non-negative matrix $A$, denote the $i$th row sum of $A$ by $r_i(A)$, and the $j$th column sum by $s_j(A).$  
Let $R=(r_1, r_2, \ldots, r_n)\in \mathbb{R}^n$, and $S=(s_1, s_2, \ldots, s_n)\in \mathbb{R}^n.$ For such a pair of vectors $R$ and $S,$ we always assume they satisfy \begin{equation} \label{con1} 0 \leq r_i,s_j \leq 1 \ \ {\rm for\ all} \ i,j=1,2,\ldots, n\end{equation} and the compatible condition
\begin{equation}\label{con2}
|R|=\sum_{i=1}^n r_i=\sum_{i=1}^n s_i=|S|.
\end{equation}
Denote by $\omega_n(R,S)$ the set of the non-negative $n\times n$ matrices with row sum vector $R$ and column sum vector $S,$  that is
\begin{equation}\label{P1}
\omega_n(R,S)=\{ A \in\omega_n\ | \ r_i(A)=r_i\  {\rm and} \ s_j(A)=s_j,\ {\rm for}\ 1\leq i,j\leq n \}.
\end{equation}
Let \begin{equation}\label{part3} \mathcal{P}_3=\{ \ \omega_n(R,S)\  | \  R\ {\rm and}\ S\ {\rm satisfying}\ \eqref{con1}\ {\rm and}\ \eqref{con2} \}.\end{equation}
Then $\mathcal{P}_3$ is a partition of $\omega_n,$ because
\begin{enumerate}[(i)]
\item $\emptyset \notin \mathcal{P}_3;$
\item $\displaystyle\bigcup_{R,S} \omega_n(R,S) =\omega_n$ where $R$ and $S$ runs over all pairs of vectors in $\mathbb{R}^n$ satisfying \eqref{con1} and \eqref{con2};
\item $\omega_n(R,S)\bigcap \omega_n(R',S')=\emptyset$ if and only if $(R,S)\neq (R',S').$
\end{enumerate}
The relations of the three partitions are shown as follows.
\begin{proposition}\label{Leiprop1} Let $\mathcal{P}_1, \mathcal{P}_2$ and $\mathcal{P}_3$ be partitions of $\omega_n$ defined by \eqref{part1}, \eqref{part2} and \eqref{part3} respectively, then
\begin{enumerate}[(i)]
\item $\mathcal{P}_2$ is a refinement of $\mathcal{P}_1.$
\item $\mathcal{P}_3$ is a refinement of $\mathcal{P}_2.$
\item $\mathcal{P}_3$ is a refinement of $\mathcal{P}_1.$
\end{enumerate}
\end{proposition}
\begin{proof}
(i) According to Theorem~\ref{thm2.1}~\cite{CKP2016}, if $B\in \omega_{n,k},$ then $\sigma(B)$ is in $[n-k,n-k+1)$ for $1\leq k \leq n.$ Therefore,
\begin{equation}\label{eq:nks}
\omega_{n,k}=\bigcup_{s\in [n-k,n-k+1)} \omega_n^s
\end{equation}
for $1\leq k \leq n.$ In particular when $k=0$, we have $$\omega_{n,0}=\omega_n^n.$$

(ii) It is due to that for all $0\leq s\leq n$,
\begin{equation}\label{eq:relsRS}
\omega_n^s=\bigcup_{|R|=|S|=s \atop 0\leq r_i, s_j \leq 1\ } \omega_n(R,S).
\end{equation}

(iii) It is a direct consequence of (i) and (ii).
\end{proof}
Moreover, every subset in each partition is convex. That means all subsets in the three partitions $\mathcal{P}_1, \mathcal{P}_2$ and $\mathcal{P}_3$ are subpolytopes of $\omega_n$. The following Proposition can be verified by direct calculation.
\begin{proposition}\label{Leiprop2}
Let $\omega_{n,k}$, $\omega_n^s$ and $\omega_n(R,S)$ be the subpolytopes of $\omega_n$ as defined by \eqref{p3}, \eqref{p2} and \eqref{P1} respectively. Then we have
\begin{enumerate}[(i)]
\item $\omega_n(R,S)$ is convex for each pair $R$ and $S$ satisfying \eqref{con1} and $\eqref{con2};$
\item $\omega_n^s$ is convex for each $0\leq s\leq n;$
\item $\omega_{n,k}$ is convex for each $k=0,1,\ldots,n.$
\end{enumerate}
\end{proposition}

An interesting fact is that the polytope $\Omega_n$ is a subpolytope of $\omega_n$ with regard to all these three partitions. In fact, $$\Omega_n=\omega_{n,0}=\omega_n^n=\omega(\mathbf{e_n},\mathbf{e_n})$$ where $\mathbf{e_n}=(1,1,1,\ldots,1)\in \mathbb{R}^n.$

\section{Extreme points of $\omega_n(R,S)$}\label{sec1}
The extreme points of transportation polytopes were characterized by Jurkat and Ryser~\cite{EB1972}. As a special case of transportation polytope, the set of extreme points of $\omega_n(R,S)$, denoted by $\mathfrak{E}_n(R,S),$ can be obtained consequently.

For $A\in \omega_n$, denote by $G_A$ the bipartite graph corresponding to $A$. The vertex set is the set of indices of the rows $I=\{i_1,i_2, \ldots i_n\}$ and columns $J=\{j_1, j_2, \ldots, j_n\}$. The edges are the places of the matrix in which the entries are positive. A {\it path} is a sequence of distinct vertices such that there exists an edge between two consecutive vertices. Also denote by  $\mathcal{B}(A)$ the $(0,1)$-matrix corresponding to $A,$ that is replacing all positive entries in $A$ by $1.$
We can then state the results about $\mathfrak{E}_n(R,S)$ as the following proposition.
\begin{proposition}\label{PropLei1}{\rm \cite{EB1972}}
If $A\in \omega_n(R,S),$ then the following conditions are equivalent:
\begin{enumerate}[(i)]
\item $A\in \mathfrak{E}_n(R,S).$
\item Every submatrix of $A$ contains a line with at most one positive entry.
\item Every submatrix $A¡ä$ of $A$ of size $m\times l$ has at most $m+l-1$ positive entries.
\item There is no matrix $B\in \omega_n(R,S)$ such that $B\neq A$ and $\mathcal{B}(B) =\mathcal{B}(A).$
\item $G_A$ is a forest with no isolated vertex.
\end{enumerate}
\end{proposition}

Denote by $\mathfrak{E}(\omega_n^s)$ the set of extreme points of $\omega_n^s$ and $\mathfrak{E}(\omega_{n,k})$ the set of extreme points of $\omega_{n,k}.$ Due to Proposition \ref{Leiprop1}, one can obtain the following inclusion relations among $\mathfrak{E}_{n}(R,S), \mathfrak{E}(\omega_n^s)$  and $\mathfrak{E}(\omega_{n,k}).$
\begin{proposition}\label{prop:increl}
\begin{enumerate} [(i)]
\item For $0 \leq s \leq n$, we have
\begin{equation*}
\mathfrak{E}(\omega_n^s)\subset\bigcup_{|R|=|S|=s \atop 0\leq r_i, s_j \leq 1\ } \mathfrak{E}_n(R,S);
\end{equation*}

\item For $0<k\leq n$, we have
\begin{equation*}
\mathfrak{E}(\omega_{n,k})\subset\bigcup_{s\in [n-k,n-k+1)} \mathfrak{E}(\omega_n^s).
\end{equation*}
\end{enumerate}
\end{proposition}

\begin{example}
Let $$A=\begin{pmatrix}0.9 & & \\ & 0.9 & \\ &&0.6 \end{pmatrix}.$$
Then $A\in \omega_3(R,S),$ where $R=S=(0.9,0.9,0.6)$. We also notice that $A\in \omega_3^{2.4}$ and $A\in \omega_{3,1}.$
According to Proposition \ref{PropLei1}, $A$ is an extreme point of $\omega_3(R,S).$ However, $A$ is not an extreme point of $\omega_3^{2.4},$ and hence not an extreme point of $\omega_{3,1}.$ Indeed,
$$A=\begin{pmatrix}0.9 & & \\ & 0.9 & \\ &&0.6 \end{pmatrix}=\frac{1}{2}\begin{pmatrix}1 & & \\ & 1 & \\ &&0.4 \end{pmatrix}+\frac{1}{2}\begin{pmatrix}0.8 & & \\ & 0.8 & \\ &&0.8 \end{pmatrix}.$$
\end{example}

\section{Extreme points of  $\omega_{n}^s$}\label{sec2}
A matrix $A$ is an extreme point of $\omega_n^s$ provided that every convex decomposition of the form
\begin{equation*}
A=\lambda A_1+(1-\lambda)A_2,\ (0\leq\lambda\leq 1)
\end{equation*}
with $A_1$ and $A_2$ in the class $\omega_n^s$ implies that $A_1=A_2=A$.
Before characterize the set of extreme points of $\omega_n^s$, i.e. $\mathfrak{E}(\omega_n^s)$, we first give some preliminaries.

For a positive integer $n$ and $0\leq s\leq n$, denote by $\lfloor s\rfloor$ the greatest integer less than or equal to $s$. Let
\begin{equation*}
v_n^s=(\underbrace{1,1,\ldots,1}_{\lfloor s\rfloor},s-\lfloor s\rfloor,0,\ldots,0)\in \mathbb{R}^n,
\end{equation*}
which contains $\lfloor s\rfloor 1$'s and satisfies $|v_n^s|=s$.
Denote by $\mathcal{RE}(v_n^s)$ the set of all rearrangements of $v_n^s$, i.e.,
\begin{equation*}
\mathcal{RE}(v_n^s)=\{v\in\mathbb{R}^n:\ \exists\ \pi\in S_n, \pi(v)=(v_{\pi(1)},v_{\pi(2)},\ldots, v_{\pi(n)})=v_n^s\}.
\end{equation*}
For $0\leq\alpha\leq 1$, let $B_m(\alpha)$ be the $m\times m$ matrix in the form:
\begin{equation} \label{leieq1}
B_m(\alpha)=\begin{pmatrix}
\alpha & 0 & \cdots & 0 \\
1-\alpha & \alpha & \cdots & 0\\
\vdots & \ddots & \ddots &\vdots\\
0      &  \cdots &  1-\alpha & \alpha
\end{pmatrix}.
\end{equation}
Notice that $\sigma(B_m(\alpha))=m-1+\alpha$, $R(B_m(\alpha))=(\alpha, 1,\ldots, 1)$ and $S(B_m(\alpha))=(1,\ldots,1,\alpha)$.
In the case $m=1$ we have $B_1(\alpha)=(\alpha)$.

\begin{theorem} \label{main}
Let $A\in \omega_n^s.$ The following statements are equivalent:
\begin{enumerate}[(a)]
\item \label{i1}$A\in \mathfrak{E}(\omega_n^s).$
\item \label{i2}There exist $R, S\in {\mathcal{RE}}(v_n^s),$ such that $A\in \mathfrak{E}_n(R,S).$
\item \label{i3}There exist $n\times n$ permutation matrices $P$ and $Q,$ such that
$$PAQ=I_{\lceil s\rceil-m}\oplus B_m(s+1-\lceil s\rceil)\oplus O_{n-\lceil s\rceil}$$
where $I_{\lceil s\rceil-m}$ is the identity matrix of size $\lceil s\rceil-m$ and $O_{n-\lceil s\rceil}$ is the zero matrix of size $n-\lceil s\rceil$.
\item \label{i4}Each connected component of $G_A$ is either an isolated vertex or a path.
For $s$ not an integer, there exists one path with length $1\leq m\leq  \lceil s\rceil$ and $\lceil s\rceil-m$ paths with length $1$ in $G_A$.
When $s$ is an integer, there are $s$ paths with length $1$ in $G_A$.
\end{enumerate}
\end{theorem}
Before the proof of Theorem \ref{main}, we provide a few examples which illustrate the idea of the proof.
\begin{example}
Let $$A=\begin{pmatrix} 0.3 & 0 & 0.4 & 0 \\ 0.2 &  \cellcolor[gray]{0.8}0.1 &0 & \cellcolor[gray]{0.8}0.2  \\ 0 & 0 & 0.3 & 0.2 \\ 0 & \cellcolor[gray]{0.8}0.1 & 0.1 & \cellcolor[gray]{0.8}0.3 \end{pmatrix}$$
where $n=4$ and $s=2.2$.
We can perturb the entries in grey a little bit such that the resultant matrix is still in $\omega_4^{2.2}$.
Notice that as long as $0<\epsilon\leq 0.1$, the following identity always holds.
$$A=\frac{1}{2}A_1+\frac{1}{2}A_2=\frac{1}{2}\begin{pmatrix} 0.3 & 0 & 0.4 & 0 \\ 0.2 & 0.1+\epsilon &0 & 0.2-\epsilon  \\ 0 & 0 & 0.3 & 0.2 \\ 0 & 0.1-\epsilon & 0.1 & 0.3+\epsilon \end{pmatrix}+\frac{1}{2}\begin{pmatrix} 0.3 & 0 & 0.4 & 0 \\ 0.2 & 0.1-\epsilon &0 & 0.2+\epsilon  \\ 0 & 0 & 0.3 & 0.2 \\ 0 & 0.1+\epsilon & 0.1 & 0.3-\epsilon \end{pmatrix},$$
where $A_1, A_2\in\omega_4^{2.2}$. This implies that $A$ is not an extreme point of $\omega_4^{2.2}$. The entries in grey form a cycle in $G_A$ and exclude the possibility for $A$ being an extreme point of $\omega_4^{2.2}$.
\end{example}



\begin{lemma} \label{lmn1}
For $0< s \leq n$ and $A\in\omega_n^s$, if the corresponding bipartite graph $G_A$ contains a cycle, then $A\notin\mathfrak{E}(\omega_n^s)$.
\end{lemma}
\begin{proof}
Let $H$ be a cycle contained in $G_A$ with edges $(i_1,j_1),(i_2,j_1),(i_2,j_2),$ $(i_3,j_2),$$\ldots,$
\\ $(i_k,j_k),(i_1,j_k)$, which means that the $(i_1,j_1),(i_2,j_1),(i_2,j_2),(i_3,j_2),$ $\ldots,(i_k,j_k),$$(i_1,j_k)$ entries of $A$ are positive. Notice that all of these entries are strictly less than one. This is because in each row of $i_1,\ldots,i_k$ and each column of $j_1,\ldots,j_k$, there are at least two positive entries. For a sufficiently small $\epsilon>0$, let $A_1$ be the matrix obtained from $A$ by adding $\epsilon$ to the entries $(i_1,j_1),\cdots,(i_k,j_k)$ and subtracting $\epsilon$ from the entries $(i_2, j_1),\cdots,(i_1,j_k)$. At the same time, let $A_2$ be the matrix obtained from $A$ by subtracting $\epsilon$ from the entries $(i_1,j_1),\cdots,(i_k,j_k)$ and adding $\epsilon$ to the entries $(i_2, j_1),\cdots,(i_1,j_k)$. Clearly, $A_1, A_2\in \omega_{n,k}$ for $\epsilon$ small enough. Actually we can see that both $A_1$ and $A_2$ have the same row and column sum vectors as $A$ does. Since $A=\frac{1}{2}A_1+\frac{1}{2}A_2$, we know that $A$ is not an extreme point.
\end{proof}

\begin{corollary}
Let $0\leq s\leq n.$ If $A\in \mathfrak{E}(\omega_n^s),$ then $G_A$ is a forest.
\end{corollary}
However, the fact that $G_A$ has no cycle does not guarantee that $A\in\mathfrak{E}(w_n^s)$.
\begin{example}\label{ex1}
Let $$A=\begin{pmatrix} 0.2 &  0.4 \\ 0.5 & 0 \end{pmatrix} \in \omega_{2}^{1.1}.$$
Clearly, there is no cycle in $G_A$. However, $A$ is not an extreme point since $$A=\frac{1}{2}\begin{pmatrix} 0.2+\epsilon &  0.4-\epsilon \\ 0.5 & 0 \end{pmatrix}+\frac{1}{2}\begin{pmatrix} 0.2-\epsilon &  0.4+\epsilon \\ 0.5 & 0 \end{pmatrix}$$ for $0<\epsilon\leq 0.2$.
\end{example}
Let $A\in \omega_n$ with an entry $a_{i,j}>0$, we call $a_{i,j}$ {\it row perturbable} if the $i$th row sum $r_i<1$. Similarly, $a_{i,j}$ is called {\it column perturbable} if the sum of the $j$th column $s_j<1$. $a_{i,j}$ is called {\it perturbable} if it is both row and column perturbable. An edge of $G_A$ is called row perturbable, column perturbable or perturbable if the corresponding entry in $A$ row perturbable, column perturbable or perturbable respectively. A path is called starting with {\it row (column) direction} if the entries corresponding to the first two edges are in the same row (column). A path is called ended in {\it row (column) direction}  if the entries corresponding to the last two edges are in the same row (column).


From Example \ref{ex1}, one can see that if a matrix containing two non-zero entries in the same row are column perturbable, then it is not an extreme point.

\begin{example} \label{ex2}
Let
$$A=\begin{pmatrix} \cellcolor[gray]{0.8} 0.2 & 0 .4  & 0 &   \cellcolor[gray]{0.8}0.4 \\ 0.4 & 0 & 0 & 0.5 \\ 0 & 0.3 & 0.5 & 0 \\ 0.3 &  \cellcolor[gray]{0.8}0.1 & 0.5 &  \cellcolor[gray]{0.8}0.1 \end{pmatrix} \in \omega_{4}^{3.7}.$$

$A$ is not an extreme point either because $$A=\frac{1}{2}\begin{pmatrix} 0.2+\epsilon & 0 .4  & 0 &   0.4-\epsilon \\ 0.4 & 0 & 0 & 0.5 \\ 0 & 0.3 & 0.5 & 0 \\ 0.3 &  0.1-\epsilon & 0.5 &  0.1+\epsilon \end{pmatrix} +\frac{1}{2}\begin{pmatrix} 0.2-\epsilon & 0.4  & 0 &   0.4+\epsilon \\ 0.4 & 0 & 0 & 0.5 \\ 0 & 0.3 & 0.5 & 0 \\ 0.3 &  0.1+\epsilon & 0.5 &  0.1-\epsilon \end{pmatrix}$$ as long as $0<\epsilon\leq 0.1$.
Note that the grey entries in $A$ are associated with the following path in the bipartite graph $G_A.$

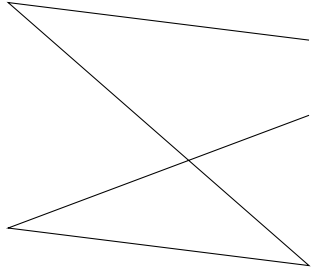
\bigskip

{ {\begin{figure}[H]
\centering \begin{tikzpicture}
\draw (4,-0.5)--(0,0)--(4,-3.5)--(0,-3)--(4,-1.5);

\put (-14,-3) {$i_1$};
\put (-2,-4) {$\bullet$}

\put (-14,-30) {$i_2$};
\put (-2,-31) {$\bullet$}

\put (-14,-60) {$i_3$};
\put (-2,-59) {$\bullet$}

\put (-14,-90) {$i_4$};
\put (-2,-88) {$\bullet$}

\put (119,-15) {$j_1$};
\put (119,-45) {$j_2$};
\put (119,-72) {$j_3$};
\put (119,-102) {$j_4$};

\put (110,-17) {$\bullet$}
\put (110,-46) {$\bullet$}
\put (110,-74) {$\bullet$}
\put (110,-102) {$\bullet$}
\end{tikzpicture}\end{figure}}}

The path is both starting and ending with the row direction. Besides, both the first edge $(i_1,j_1)$ and the last edge $(i_4,j_2)$ are column perturbable.
\end{example}

Indeed, for $A\in \omega_n^s$, if there exists a path in the bipartite graph $G_A$
$$v_{j_1}\rightarrow v_{i_1}\rightarrow v_{j_2}\rightarrow \cdots \rightarrow v_{j_k}\rightarrow v_{i_k} \rightarrow v_{j_{k+1}}$$ containing even number of edges, then we can perturb the edges by adding a small positive number $\epsilon$ to the entry corresponding to the first edge and subtract an $\epsilon$ from the entry corresponding to the second edge and so on to get a matrix $A_1$. We illustrate this process by the following graph
\begin{equation*}
v_{j_1}{\longrightarrow^{\!\!\!\!\!\!\!\!\!\!+\epsilon}}\ v_{i_1} {\longrightarrow^{\!\!\!\!\!\!\!\!\!\!-\epsilon}}\  v_{j_2}{\longrightarrow^{\!\!\!\!\!\!\!\!\!\!+\epsilon}}\  \cdots {\longrightarrow^{\!\!\!\!\!\!\!\!\!\!-\epsilon}}\  v_{j_k}{\longrightarrow^{\!\!\!\!\!\!\!\!\!\!+\epsilon}}\  v_{i_k} {\longrightarrow^{\!\!\!\!\!\!\!\!\!\!-\epsilon}}\  v_{j_{k+1}}.
\end{equation*}
Similarly we can do the opposite operation to get another matrix $A_2$, which can be illustrated by the following graph
\begin{equation*}
v_{j_1}{\longrightarrow^{\!\!\!\!\!\!\!\!\!\!-\epsilon}}\ v_{i_1} {\longrightarrow^{\!\!\!\!\!\!\!\!\!\!+\epsilon}}\  v_{j_2}{\longrightarrow^{\!\!\!\!\!\!\!\!\!\!-\epsilon}}\  \cdots {\longrightarrow^{\!\!\!\!\!\!\!\!\!\!+\epsilon}}\  v_{j_k}{\longrightarrow^{\!\!\!\!\!\!\!\!\!\!-\epsilon}}\  v_{i_k} {\longrightarrow^{\!\!\!\!\!\!\!\!\!\!+\epsilon}}\  v_{j_{k+1}}.
\end{equation*}
Then we have $\sigma(A_1)=\sigma(A_2)=\sigma(A)=s$ and $A=\frac{1}{2}(A_1+A_2)$. The only question is whether $A_1$ and $A_2$ are still doubly substochastic matrices. Notice that when we add an $\epsilon$ to the first edge or to the last edge, the corresponding column sums might be greater than one. To make sure that $A_1$ and $A_2$ are still in $\omega_n$, $b_{i_1, j_1}$ and $b_{i_{k}, j_{k+1}}$ must be column perturbable. In Example \ref{ex2}, both $(i_1,j_1)$ entry and $(i_4,j_2)$ entry are column perturbable, making the construction of $A_1$ and $A_2$ valid.

\begin{proposition} \label{prop1}
Let $\mathcal{P}$ be a path in the bipartite graph $G_A.$ Then $\mathcal{P}$ contains even number of edges if
and only if the starting direction and the ending direction of $\mathcal{P}$ are the same.
\end{proposition}
\begin{proof} Denote the edges on $\mathcal{P}$ by ${e_1},{e_2},\ldots,{e_k}.$
Both directions follow that if the entries corresponding to $e_i$ and $e_{i+1}$ are in the same row (column) and then the entries corresponding to $e_{i+1}$ and $e_{i+2}$ are in the same column (row) for $i=1,2, \ldots, k-2.$
\end{proof}

\begin{lemma} \label{lm2}
Let $A\in \omega_n^s.$ $A$ is not an extreme point of $\omega_n^s$, if $G_A$ contains a path $\mathcal{P}$ with even number of edges satisfying one of the following:
\begin{enumerate}[(i)]
\item $\mathcal{P}$ starts and ends in row direction and the entries corresponding to the first edge and the last edge are column perturbable; or
\item $\mathcal{P}$ starts and ends in column direction and the entries corresponding to the first edge and the last edge are row perturbable.
\end{enumerate}
\end{lemma}

\begin{proof} (i). According to Proposition \ref{prop1},  $\mathcal{P}$ contains even number of edges. Here is a method to construct two matrices $A_1$ and $A_2$ such that $A=\frac{1}{2}(A_1+A_2)$. To construct $A_1$, one may first add a positive number $\epsilon$ to the entry corresponding to the first edge. Then subtract an $\epsilon$ from the entry corresponding to the second edge. Then again add an $\epsilon$ to the entry corresponding to the third edge and subtract an $\epsilon$ from the entry corresponding to the forth edge. Keep adding and subtracting an $\epsilon$ alternately until we finish subtracting from the last entry corresponding to the last edge. Thus we obtain the matrix $A_1$. Since $\mathcal{P}$ contains even number of edges, $\sigma(A_1)=\sigma(A).$ To construct $A_2,$ simply switch additions and subtractions in the constriction of $A_1$. Clearly $\sigma(A_2)=\sigma(A)$ and $A=\frac{1}{2}(A_1+A_2).$ Since both entries corresponding to the first edge and the last edge are column perturbable and the path $\mathcal{P}$ starts and ends in row direction, both $A_1$ and $A_2$ are doubly substochastic matrices for $\epsilon$ small enough.

(ii). By taking the transpose of $A$, it is converted to Case (i).
\end{proof}

\begin{lemma} \label{lm3}
Let $A\in \omega_n^s.$ $A$ is not an extreme point of $\omega_n^s$ if $G_A$ contains a path $\mathcal{P}$ satisfying
\begin{enumerate}[(i)]
\item both the first edge and the last edge are column perturbable; or
\item both the first edge and the last edge are row perturbable.
\end{enumerate}
\end{lemma}
\begin{proof} (i). If $\mathcal{P}$ starts and ends in row direction, then Lemma \ref{lm3} holds. Suppose $\mathcal{P}$ starts in column direction.  Note that $\mathcal{P}$ contains odd number of edges and the entry in $A$ corresponding to the second edge of $\mathcal{P}$ is in the same column as the entry corresponding to the first edge, so it is column perturbable too. Let $\mathcal{P'}$ be the path by removing the first edge from $\mathcal{P}.$ Then $\mathcal{P'}$ starts and ends in row direction and the entries corresponding to the first edge and the last edge are column perturbable. Thus $A$ is not an extreme point. If $\mathcal{P}$ ends in column direction, then construct $\mathcal{P'}$ by removing the last edge from $\mathcal{P}.$

(ii). By taking the transpose of $A$, it is converted to Case (i).
\end{proof}

\begin{proposition} \label{prop1}
Let $A\in \omega_n^s.$
\begin{enumerate}[(i)]
\item If there exists a path $\mathcal{P}$ in which the first edge and the last edge are column perturbable, then there exists a path $\mathcal{P'}$ which starts and ends in row direction and in which the first edge and the last edge are column perturbable.
\item  If there exists a path $\mathcal{P}$ in which the first edge and the last edge are row perturbable, then there exists a path $\mathcal{P'}$ which starts and ends in column direction and in which the first edge and the last edge are row perturbable.
\item If there exists a path $\mathcal{P}$ in which the first edge is column perturbable and the last edge is row perturbable, then there exists a path $\mathcal{P'}$ which starts in row direction and ends in column direction and in which the first edge is column perturbable and the last edge is row perturbable.
\item If there exists a path $\mathcal{P}$ in which the first edge is row perturbable and the last edge is column perturbable, then there exists a path $\mathcal{P'}$ which starts in column direction and ends in row direction and in which the first edge is row perturbable and the last edge is column perturbable.
\end{enumerate}

\end{proposition} \label{propn1}
\begin{proof} (i). Let $\mathcal{P}$ be the path in which the first edge and the last edge are column perturbable. Then $\mathcal{P}$ contains even number of edges and hence $\mathcal{P}$ starts and ends in the same direction. If $\mathcal{P}$ starts and ends in row direction, then we are done. Suppose that $\mathcal{P}$ starts in column direction, which means that the entry corresponding to the second edge is in the same column as the entry corresponding to the first edge. Hence it is column perturbable. Similarly the entry corresponding to the last second edge is column perturbable as well. Then one can construct $\mathcal{P'}$ by removing the first and the last edge.

(ii). By taking the transpose of $A$, it is converted to Case (i).

(iii). By the same argument in (i), one may remove the first edge and the last edge if the given path does not start and end in desired directions.

(iv). By taking the transpose of $A$, it is converted to Case (iii).
\end{proof}

\begin{example}
Let $A=\begin{pmatrix}0 & 0.6 & 0.4 \\ 0.6 & 0.4 & 0 \\ 0.4 & 0 & 0\end{pmatrix}\in \omega_3^{2.4}.$ Although $G_A$ has a path in which the first edge, connecting $j_3$ and $i_1,$ is column perturbable and the last edge, connecting $j_1$ and $i_3,$ is row perturbable, it is an extreme point of $\omega_3^{2.4}.$

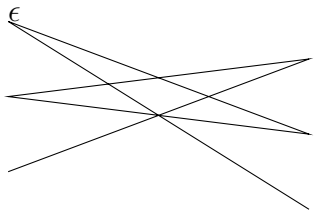
{ {\begin{figure}[H]
\centering \begin{tikzpicture}
\draw (0,-2)--(4,-0.5)--(0,-1)--(4,-1.5)--(0,0)--(4,-2.5);

\put (-14,-3) {$i_1$};
\put (-2,-4) {$\bullet$}

\put (-14,-30) {$i_2$};
\put (-2,-31) {$\bullet$}

\put (-14,-60) {$i_3$};
\put (-2,-59) {$\bullet$}

\put (119.5,-15) {$j_1$};
\put (119.5,-45) {$j_2$};
\put (119.5,-72) {$j_3$};

\put (110,-18) {$\bullet$}
\put (110,-45.2) {$\bullet$}
\put (110,-74) {$\bullet$}
\end{tikzpicture}\end{figure}}}

\bigskip
\end{example}

\begin{theorem} \label{thm1}
Let $n$ be a positive integer and $0\leq s\leq n.$ Let $A\in \omega_n^s.$ If $A$ has two rows or two columns whose sum are strictly greater than $0$ and strictly less than $1$, then $A$ is not an extreme point of $\omega_n^s.$
\end{theorem}
\begin{proof}
Without loss of generality, we can assume that the $s$th row and $t$th row sums are strictly between $0$ and $1.$ Otherwise, one may take the transpose of $A$.

If there is a path in $G_A$ connecting vertices $i_s$ and $i_t$, then by Lemma~\ref{lm3}, Theorem~\ref{thm1} holds.

Suppose the vertices $i_s$ and $i_t$ are in the different connected components $G_s$ and $G_t$ respectively. If either $G_s$ or $G_t$ contains a cycle, then $A$ is not an extreme point due to Lemma \ref{lmn1}. Suppose neither $G_s$ nor $G_t$ contains a cycle. Since $G_s$ does not contain a cycle and is connected, there exists a vertex $v$ with degree $1$ which implies that the edge connected to $v$ is either row perturbable or column perturbable. Since $i_s$ is row perturbable, if $v$ is row perturbable as well, then $A$ is not an extreme point due to Lemma \ref{lm3}. Suppose the edge connected to $v$ is column perturbable, then there exists a path $\mathcal{P}_s$ connecting $i_s$ and $v$ in which the first edge is row perturbable and the last edge is column perturbable. Also it starts in column direction and ends in row direction. Similarly, there exists a path $\mathcal{P}_t$ in $G_t$ connecting $i_t$ and a vertex $w$ which has degree $1$ and the only edge connecting to $w$ is column perturbable.  Note that both $\mathcal{P}_s$ and $\mathcal{P}_t$ have odd number of edges. To construct $A^+,$ first adding an $\epsilon$ to the entry corresponding to the first edge in $\mathcal{P}_s,$ and then subtracting an $\epsilon$ from the entry corresponding to the second edge in $\mathcal{P}_s,$  and then adding an $\epsilon$ to the entry corresponding to the third edge in $\mathcal{P}_s,$ and so on until adding or subtracting an $\epsilon$ from all edges in $\mathcal{P}_s$ alternatively. Meanwhile subtracting an $\epsilon$ from the entry corresponding to the first edge in $\mathcal{P}_t,$  and then adding an $\epsilon$ to the entry corresponding to the second edge in $\mathcal{P}_t$, and so on until subtracting or adding an $\epsilon$ on all edges in $\mathcal{P}_t.$ Then by switching all additions and subtractions we get $A^-$. Note that if $\epsilon>0$ is small enough, then $A^+, A^-\in \omega_n^s$ and $A=\frac{1}{2}(A^++A^-).$

\end{proof}

\begin{corollary}\label{Lco1}
Let $A\in \omega_n^s$, and $R(A),S(A)$ be the row sum vector and column sum vector respectively. If $A$ is an extreme point of $\omega_n^s$, then $R(A), S(A)\in \mathcal{RE}(v_n^s).$
\end{corollary}


\begin{lemma}
Let $n$ be a positive integer and $n-1 < s \leq n.$  $B_n(s-n+1)$ as defined in \eqref{leieq1} is an extreme point of $\omega_n^s.$
\end{lemma}

\begin{proof}Suppose $B_n(\alpha)$ is not an extreme point, then there exist $A$ and $B \in \omega_n^s,$ such that
$$B_n(\alpha)=cA+(1-c)B$$ for some $0<c<1.$ Note that both $A$ and $B$ must be in the form
$$\begin{pmatrix}\ x & * & &&& \\ & * & * &&& \\ && * & * && \\ &&& \vdots & \cdots && \\ &&&& \vdots & \cdots & \\ &&&&& * & * \\ &&&&&& *\end{pmatrix}$$
where $x$ and $*$ are all possible positive elements. Secondly, since each row sum of $B_n(\alpha)$ is $1$ except the last one, and each column sum of $B_n(\alpha)$ is $1$ except the first one, $A$ and $B$ must have sum $1$ in these rows and columns meaning that the sum of all elements in $A$ is $n-1+x.$ since $A\in \omega_n^s,$ so $x=s-n+1$ which implies that $A=B_n(\alpha)$, where $\alpha=s-n+1$.
\end{proof}

\begin{corollary}
Let $A\in \omega_n^s$ where $n-1< s\leq n,$ and $\alpha=s-n+1.$ Then $A$ is an extreme point if there exist $n\times n$ permutation matrices $P$ and $Q,$ such that $$PAQ=B_n(\alpha).$$
\end{corollary}

\begin{corollary}\label{Lco2} Let $A\in \omega_n^s$ where $n-1< s\leq n$. $A$ is an extreme point of $\omega_n^s,$ if there exist $n\times n$ permutation matrices $P$ and $Q,$ such that
$$PAQ=I_{n-m}\oplus B_m(s-n+1)$$
for some $0\leq m\leq n,$ where $I_{n-m}$ is the identity matrix with order $n-m.$
\end{corollary}
\begin{proof}
Because both $I_{n-m}$ and $ B_m(s-n+1)$ are uniquely determined by themselves, $PAQ=I_{n-m}\oplus B_m(s-n+1)$ is an extreme point.
\end{proof}

\noindent \textit{Proof of Theorem \ref{main}.} \ $\eqref{i1}\Rightarrow\eqref{i2}.$ It is due to Corollary \ref{Lco1}.

$\eqref{i2}\Rightarrow\eqref{i3}.$ This is due to the method of constructing elements in $\mathfrak{E}_n(R,S)$ provided in Section 2, \cite{TR1}.


$\eqref{i3}\Rightarrow\eqref{i1}.$ It is due to Corollary \ref{Lco2}.

$\eqref{i3}\Leftrightarrow\eqref{i4}.$ It is trivial.

\section{Another approach to the extreme points of $\omega_n^s$}\label{sec3}
In this section, we develop an algorithm to find extreme points of $\omega_n^s,$ through which one may express a matrix in $\mathfrak{E}_n(R,S)$ as a convex combination of matrices in $\mathfrak{E}(\omega_n^s)$. It can be treated as an alternative proof of Theorem~\ref{main}. According to Proposition~\ref{prop:increl}, recall that
\begin{equation*}
\mathfrak{E}(\omega_n^s)\subseteq\bigcup_{|R|=|S|=s \atop 0\leq r_i, s_j \leq 1\ } \mathfrak{E}_n(R,S).
\end{equation*}
In this section we will give more clear relation between $\mathfrak{E}_n(R,S)$ and $\mathfrak{E}(\omega_n^s)$.
Based on this result the relation between $\mathfrak{E}(\omega_n^s)$ and $\mathfrak{E}(\omega_{n,k})$ will be given in section~\ref{sec4}.


We first describe the inductive procedure given by Jurkat and Ryser~\cite{TR1} to construct the extreme points of $\omega_n(R,S)$ for given $R$ and $S$. To obtain an arbitrary matrix $A\in\mathfrak{E}_n(R,S)$, we select a position $(i, j)$ in $A$ and define
\begin{equation*}
a_{i,j}=\min\{r_i, s_j\}.
\end{equation*}

If $r_i\leq s_j$, then we can complete row $i$ by inserting $(n-1)\ 0$'s. The submatrix $A_1$ obtained from $A$ by deleting row $i$ is then a matrix of size $(n-1)\times n$ whose row sum vector is denoted by
\begin{equation}\label{eq:rowsvec}
R_1=(r_1,\ldots,r_{i-1},\times,r_{i+1},\ldots,r_n).
\end{equation}
Here the $"\times"$ symbol at the $i$th entry means that we have finished the construction of row $i$.
The compatible column sum vector should be
\begin{equation}\label{eq:colsvec}
S_1=(s_1,\ldots,s_{j-1},s_j-r_i,s_{j+1},\ldots,s_n).
\end{equation}
Similarly if $r_i\geq s_j$, then we complete column $j$ by inserting $(n-1)\ 0$'s. We require that the submatrix $A_1$ obtained from $A$ by deleting column $j$ be a matrix of size $n\times (n-1)$ with row sum vector
\begin{equation}\label{eq:colsvec2}
R_1=(r_1,\ldots,r_{i-1},r_i-s_j,r_{i+1},\ldots,r_n).
\end{equation}
The compatible column sum vector should be
\begin{equation}\label{eq:rowsvec2}
S_1=(s_1,\ldots,s_{j-1},\times,s_{j+1},\ldots,s_n).
\end{equation}
In case that $r_i=s_j$ either of the above constructions is allowed. Keep doing this process and we can get all elements in $A$.

To illustrate this process more clearly, we can make use of the weighted directed bipartite graphs. Associated with each $n\times n$ matrix $A$ in $\mathfrak{E}(R,S)$, there is a weighted bipartite graph $\mathcal{G}(A)$ with vertex set $\{r_1, r_2,\ldots, r_n\}\cup\{s_1, s_2,\ldots, s_n\}$. To obtain $\mathcal{G}(A)$, we simply follow the construction procedure. For the selected position $(i,j)$ in $A$, we compare the value of $r_i$ and $s_j$. If $r_i\leq s_j$, there is a directed edge from $r_i$ to $s_j$ with weight $r_i$. If $r_i\geq s_j$, there is a directed edge from $s_j$ to $r_i$ with weight $s_j$. If $r_i=s_j$ then we can do either one of the above steps. Then we do the same way for the row and column sum vector $R_1$ and $S_1$. Keep doing this process and by induction we can get $\mathcal{G}(A)$. Notice that the edges constructed in this way have an order which coincides with the order of the construction of elements in $A$.

\begin{example}
Let $R=(0.6, 0.9,0.7,0.4,0.8)$ and $S=(0.8,0.7,0.9,0.6,0.4)$. We first compare $r_1=0.6$ and $s_2=0.8$ to get $a_{1,1}=0.6$ and $a_{1,2}=a_{1,3}=a_{1,4}=a_{1,5}=0$. The new row and column sums are $R_1=(\times, 0.9,0.7,0.4,0.8)$ and $S_1=(0.2,0.7,0.9,0.6,0.4)$. Then we compare $r_4=0.4$ and $s_2=0.7$ to get $a_{4,2}=0.4$ and $a_{4,1}=a_{4,3}=a_{4,4}=a_{4,5}=0$. The new row and column sums are $R_2=(\times, 0.9,0.7,\times,0.8)$ and $S_1=(0.2,0.3,0.9,0.6,0.4)$. Then we compare $r_2=0.9$ and $s_2=0.3$ to get $a_{2,2}=0.3$ and $a_{3,2}=a_{5,2}=0$. The new row and column sums are $R_3=(\times, 0.6,0.7,\times,0.8)$ and $S_1=(0.2,\times,0.9,0.6,0.4)$. We compare $r_2=0.6$ and $s_3=0.9$ to get $a_{2,3}=0.6$ and $a_{2,1}=a_{2,4}=a_{2,5}=0$. The new row and column sums are $R_4=(\times, \times ,0.7,\times,0.8)$ and $S_1=(0.2,\times,0.3,0.6,0.4)$. We compare $r_3=0.7$ and $s_3=0.3$ to get $a_{3,3}=0.3$ and $a_{5,3}=0$. The new row and column sums are $R_5=(\times, \times ,0.4,\times,0.8)$ and $S_1=(0.2,\times,\times,0.6,0.4)$. We compare $r_3=0.4$ and $s_4=0.6$ to get $a_{3,4}=0.4$ and $a_{3,1}=a_{3,5}=0$. The new row and column sums are $R_5=(\times, \times ,\times,\times,0.8)$ and $S_1=(0.2,\times,\times, 0.2,0.4)$. We compare $r_5=0.8$ and $s_4=0.2$ to get $a_{5,4}=0.2$. The new row and column sums are $R_5=(\times, \times ,\times,\times,0.6)$ and $S_1=(0.2,\times,\times,\times,0.4)$. We compare $r_5=0.6$ and $s_1=0.2$ to get $a_{5,1}=0.2$ and $a_{4,1}=0$. The new row and column sums are $R_5=(\times, \times ,\times,\times,0.4)$ and $S_1=(\times,\times,\times,\times,0.4)$, which implies $a_{5,5}=0.4$. The matrix is
\begin{equation*}
\begin{pmatrix}
0.6 & 0 & 0 & 0 & 0 \\
0  & 0.3 & 0.6 & 0 & 0\\
0  & 0 & 0.3  & 0.4 & 0\\
0  & 0.4 & 0 & 0 & 0\\
0.2 & 0 & 0 & 0.2 & 0.4
\end{pmatrix}.
\end{equation*}
The corresponding weighted directed bipartite graph is





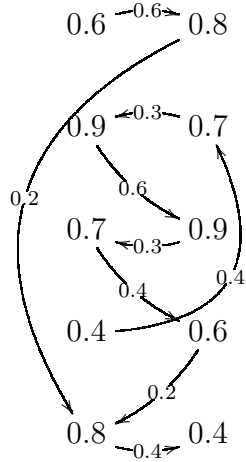
\begin{equation*}
\xymatrix{
0.6 \ar@/^/[r]|-{0.6}  & 0.8 \ar@/_4pc/[ddddl]|-{0.2} \\
0.9 \ar@/_/[dr]|-{0.6} & 0.7 \ar@/_/[l]|-{0.3} \\
0.7 \ar@/_/[dr]|-{0.4} & 0.9 \ar@/^/[l]|-{0.3}\\
0.4 \ar@/_3pc/[uur]|-{0.4} & 0.6\ar@/^/[dl]|-{0.2}\\
0.8  \ar@/_/[r]|-{0.4}& 0.4 }
\end{equation*}
\end{example}

From the bipartite graph, we know that permuting the entries in the row or column sum vector does not impact the constructing of the row and column elements in $A$. Thus the order of the row sum elements in $R$ or the order of the column sum elements in $S$ does not impact the construction of the matrix.
\begin{lemma}
Let $R=(r_1,r_2,\ldots,r_n)$ be the row sum vector and $S=(s_1,s_2,\ldots,s_n)$ be the column sum vector. Both $R$ and $S$ satisfy the compatible condition $|R|=|S|$. For any $\pi$ ranging over all permutations of $1,\ldots,n$, let $\pi(R)=(r_{\pi(1)},r_{\pi(2)},\ldots, r_{\pi(n)})$. Similarly, for any $\tau$ ranging over all permutations of $1,\ldots,n$, let $\tau(S)=(s_{\tau(1)},s_{\tau(2)},\ldots,s_{\tau(n)})$. Then up to row and column permutations,
\begin{equation}\label{eq:perex}
\mathfrak{E}(R,S)=\mathfrak{E}(\pi(R),\tau(S)).
\end{equation}
\end{lemma}


Recall that for $0\leq\alpha\leq 1$, the matrix $B_m(\alpha)$ is the $m\times m$ matrix as follows:
\begin{equation*}
B_m(\alpha)=\begin{pmatrix}
\alpha & 0 & \cdots & 0 \\
1-\alpha & \alpha & \cdots & 0\\
\vdots & \ddots & \ddots &\vdots\\
0      &  \cdots &  1-\alpha & \alpha
\end{pmatrix}.
\end{equation*}
\begin{theorem}\label{thm:stanext}
$B_m(\alpha)$ is an extremal matrix of $\omega_m^{m-1+\alpha}$ for all $m\geq 1$.
\end{theorem}
\begin{proof}
Assume that there exists $A_1, A_2\in\omega_m^{m-1+\alpha}$ such that
\begin{equation*}
B_m(\alpha)=\lambda A_1+(1-\lambda) A_2,
\end{equation*}
where $0<\lambda<1$. Notice that if the $(i,j)$ entry of $B_m(\alpha)$ is equal to $0$ or $1$, then that entry of $A_1$ and $A_2$ should also be the same value as in $A$. All the columns of $B_m(\alpha)$ from the first to the last second have the same column sum $1$, which implies that the columns of both $A_1$ and $A_2$ must also satisfy this property. For the same reason the second until the last row sums of $A_1$ and $A_2$ are all equal to $1$. Therefore, $A_1=A_2=B_m(\alpha)$.
\end{proof}
The following theorem gives a description of the extremal matrices of $\omega_n^s(R_{n_0},S_{n_0})$, where $R_{n_0}=S_{n_0}=v_n^s=(1,\ldots,1,s-\lfloor s\rfloor,0,\ldots,0)$.
\begin{theorem}\label{thm:extrs}
An $n\times n$ matrix $A$ is an extremal matrix of the set $\omega_n^s(R_{n_0},S_{n_0})$ where $R_{n_0}=S_{n_0}=v_n^s=(1,\ldots,1,s-\lfloor s\rfloor, 0,\ldots,0)$ if and only if $A$ can be permuted into the following direct sum form
\begin{equation}\label{eq:extmat}
I_{\lceil s\rceil-m}\oplus B_m(s+1-\lceil s\rceil) \oplus O_{n-\lceil s\rceil}
\end{equation}
for some nonnegative integer $0\leq m\leq \lceil s\rceil$.
\end{theorem}
\begin{proof}
Assume that $A$ can be permuted into the form as in~(\ref{eq:extmat}). If there exists $A_1, A_2\in\omega_n^s(R_{n_0},S_{n_0})$ such that
\begin{equation*}
I_{\lceil s\rceil-m}\oplus B_m(s+1-\lceil s\rceil) \oplus O_{n-\lceil s\rceil} =\lambda A_1+(1-\lambda) A_2,
\end{equation*}
where $0<\lambda<1$.
Then we can write $A_1=I_{\lceil s\rceil-m}\oplus A_1' \oplus O_{n-\lceil s\rceil}$ and $A_2=I_{\lceil s\rceil-m}\oplus A_2' \oplus O_{n-\lceil s\rceil}$.
By Theorem~\ref{thm:stanext}, $B_{m}(s+1-\lceil s\rceil) $ is extremal, which implies that $A_1'=A_2'=B_{m}(s+1-\lceil s\rceil)$. Therefore $A\in\mathfrak{E}_n(R_{n_0},S_{n_0})$.

For each $A\in\mathfrak{E}_n(R_{n_0},S_{n_0})$ , we prove by induction that $A$ can be permuted into the form $I_{\lceil s\rceil-m}\oplus B_m(s+1-\lceil s\rceil) \oplus O_{n-\lceil s\rceil}$. According to the construction of matrices in $\mathfrak{E}_n(R_{n_0},S_{n_0})$ given by Jurkat and Ryser~\cite{TR1}, we first select a position $(i,j)$ in $A$ and let $a_{i,j}=\min\{r_i, s_j\}$. If both $r_i$ and $s_j$ are equal to $1$, then $a_{i,j}=1$. In this case, we complete row $i$ and column $j$ by inserting $(2n-1)\ 0$'s. Let $A_1$ be the submatrix obtained from $A$ by removing row $i$ and column $j$. Thus, the row and column sum vector of $A_1$, denoted by $R_{n_1}$ and $S_{n_1}$, are just obtained from $R_{n_0}$ and $S_{n_0}$ by removing $r_i$ and $s_j$ respectively. Compared with $R_{n_0}$ and $S_{n_0}$, both $R_{n_1}$ and $S_{n_1}$ of $A_1$ contain one less entry $1$. Thus $A_1$ must be in the class of the extremal matrices $\mathfrak{E}_{n-1}(R_{n_1},S_{n_1})$. By the induction hypothesis, $A_1$ can be permuted into the direct sum form as in~(\ref{eq:extmat}), which implies that $A$ satisfies the assertion. The case when both $r_i$ and $s_j$ are equal to $0$ or $s+1-\lceil s\rceil$ can be proved similarly.

Next we consider the cases when both $r_i$ and $s_j$ are not equal to zero, which means either $r_i=s+1-\lceil s\rceil, r_j=1$ or $r_i=1, r_j= s+1-\lceil s\rceil$. We only need to consider the case when $r_i=s+1-\lceil s\rceil, r_j=1$. Otherwise we can interchange the rows and columns by transposing the matrix. Thus $a_{i,j}=\min\{r_i,s_j\}=s+1-\lceil s\rceil$ and complete row $i$ by inserting $(n-1)\ 0$'s. Removing row $i$ from $A$ we get a matrix $A_1$ of size $(n-1)\times n$ with row sum vector $R_1$ and column sum vector $S_1$, as described in the construction of $A$. We may write
\begin{equation*}
R_1=(r_1,\ldots,r_{i-1},\times,r_{i+1},\ldots,r_n),
\end{equation*}
where each entry in $R_1$ is either $1$ or $0$.
The compatible column sum vector should be
\begin{equation*}
S_1=(s_1,\ldots,s_{j-1}, \lceil s\rceil-s,s_{j+1},\ldots,s_n).
\end{equation*}
Therefore, the only nonzero element in the $j$th column of $A_1$ besides $a_{i,j}$ is $\lceil s\rceil-s$. Removing column $j$ from $A_1$ we get an $(n-1)\times (n-1)$ matrix $A_2$ whose row and column sum vectors can be rearranged into an $(n-1)$-dimensional vector $(1,\ldots,1,s-\lfloor s\rfloor,0,\ldots,0)$ with $(\lfloor s\rfloor-1) 1$'s. By induction hypothesis, the assertion holds.

Suppose that either $r_i$ or $s_j$ is equal to zero, simply follow the contruction until both the row sum element and the column sum element compared in the pair are nonzero. This then can be reduced into one of the above cases discussed before. Thus we prove the theorem.
\end{proof}

\begin{corollary}\label{cor:ext}
If $A\in\mathfrak{E}_n(R_{n_0},S_{n_0})$ where $R_{n_0}=S_{n_0}=(1,\ldots,1,s-\lfloor s\rfloor, 0,\ldots,0)$, then $A\in\mathfrak{E}(\omega_n^s)$.
\end{corollary}
\begin{proof}
If $A\in\mathfrak{E}_n(R_{n_0},S_{n_0})$  where $R_{n_0}=S_{n_0}=(1,\ldots,1,s-\lfloor s\rfloor, 0,\ldots,0)$, then by Theorem~\ref{thm:extrs}, $A$ can be permuted into the form $I_{\lceil s\rceil-m}\oplus B_m(s+1-\lceil s\rceil) \oplus O_{n-\lceil s\rceil}$ for some nonnegative integer $0\leq m\leq \lceil s\rceil$. Suppose that there exist $B, C\in\omega_n^s$ such that
\begin{equation*}
I_{\lceil s\rceil-m}\oplus B_m(s+1-\lceil s\rceil) \oplus O_{n-\lceil s\rceil}=\lambda B +(1-\lambda)C,
\end{equation*}
where $0<\lambda<1$. Then both $B$ and $C$ must be in the form
\begin{align*}
B&=I_{\lceil s\rceil-m}\oplus B_1 \oplus O_{n-\lceil s\rceil},\\
C&=I_{\lceil s\rceil-m}\oplus C_1 \oplus O_{n-\lceil s\rceil},
\end{align*}
where $B_1, C_1$ are two $m\times m$ matrices.
According to Theorem~\ref{thm:stanext}, since $B_m(s+1-\lceil s\rceil)$ is extremal, we know that $B_1=C_1=B_{m}(s+1-\lceil s\rceil)$.
\end{proof}
\begin{example}
Let $n=4, s=3.6$, then we can find all extremal matrices with row and column sums equal to $(1,1,1,0.6)$ up to permutations of row and columns as follows:
\begin{equation*}
\begin{pmatrix}
1 & 0 & 0 & 0 \\
0 & 1 & 0 & 0 \\
0 & 0 & 1 & 0 \\
0 & 0 & 0 & 0.6
\end{pmatrix},
\begin{pmatrix}
1 & 0 & 0 & 0 \\
0 & 1 & 0 & 0 \\
0 & 0 & 0.6 & 0 \\
0 & 0 &  0.4 & 0.6
\end{pmatrix},
\begin{pmatrix}
1 & 0 & 0 & 0 \\
0 & 0.6 & 0 & 0\\
0 & 0.4 & 0.6 & 0\\
0 &  0 & 0.4 & 0.6
\end{pmatrix},
\begin{pmatrix}
0.6 & 0 & 0 & 0 \\
0.4 & 0.6 & 0 & 0\\
0 & 0.4 & 0.6 & 0\\
0 & 0 & 0.4 & 0.6
\end{pmatrix}.
\end{equation*}
It is easy to see that they are also extremal matrices of $\omega_4^{3.6}$.
\end{example}

Moreover, we can get the following theorem which gives us a description of all the extremal matrices of the set $\omega_n^s$.
\begin{theorem}\label{thm:iffext}
$A\in\mathfrak{E}(\omega_n^s)$ if and only if $A\in\mathfrak{E}_n(R_{n_0},S_{n_0})$, where both $R_{n_0}$ and $S_{n_0}$ can be permuted into $(1,\ldots,1,s-\lfloor s\rfloor, 0,\ldots,0)$.
\end{theorem}
\begin{proof}
By Corollary~\ref{cor:ext}, we only need to show that if $A\in\mathfrak{E}(\omega_n^s)$, then both the row sum vector $R(A)$ and the column sum vectors $S(A)$ can be reordered into the vector $(1,\ldots,1,s-\lfloor s\rfloor, 0,\ldots,0)$ with $\lfloor s\rfloor\ 1$'s and $(n-\lceil s\rceil)\ 0$'s. Notice that $A\in\mathfrak{E}(\omega_n^s)$ implies that $A\in\mathfrak{E}_n(R(A),S(A))$, where $R(A)$ and $S(A)$ denote the row and column sum vectors respectively.  Without loss of generality we assume that $A\in\mathfrak{E}_n(R(A),S(A))$ and $R(A)$ cannot be rearranged into $(1,\ldots,1,s-\lfloor s\rfloor, 0,\ldots,0)$. Otherwise we can transpose $A$ and thus interchange $R(A)$ and $S(A)$. We want to show that $A$ is not an extreme point of $\omega_n^s$. Then there exist at least two nonzero entries $r_i$ and $r_{i'}$ such that $r_i, r_{i'}<1$. Two vertices are said to be in the same connected component if there exist some edges connecting these two vertices regardless of the directions.

Case 1. If $r_i$ and $r_{i'}$ are not in the same connected component, then there exist at least two entries $s_j, s_{j'}$ in $S(A)$, which are in the same connected component with $r_i$ and $r_{i'}$ respectively, and both $s_j$ and $s_{j'}$ are less than one. Consider
\begin{align*}
R'&=(r_1, \ldots, r_i+\epsilon_1,\ldots, r_{i'}-\epsilon_1,\ldots,r_n),\\
S'&=(s_1, \ldots, s_j+\epsilon_1,\ldots, s_{j'}-\epsilon_1,\ldots,s_n),
\end{align*}
and
\begin{align*}
R''&=(r_1, \ldots, r_i-\epsilon_2,\ldots, r_{i'}+\epsilon_2,\ldots,r_n),\\
S''&=(s_1, \ldots, s_j-\epsilon_2,\ldots, s_{j'}+\epsilon_2,\ldots,s_n).
\end{align*}
Notice that $|R'|=|R''|=|R(A)|$ and $|S'|=|S''|=|S(A)|$. For sufficiently small $\epsilon_1, \epsilon_2>0$, we can always find $A_1\in\omega_n(R',S')$ and $A_2\in\omega_n(R'',S'')$ such that
\begin{equation*}
A=\lambda A_1+(1-\lambda)A_2, \ \rm{where}\ \lambda=\frac{\epsilon_2}{\epsilon_1+\epsilon_2}>0.
\end{equation*}

Case 2. Suppose $r_i$ and $r_{i'}$ are in the same connected component. Without loss of generality, according to the order of construction of $A$, we assume that $r_i$ is the first row sum entry strictly less than one with an edge from $r_i$ to a column sum entry $s_j$. We divide the case into three subcases as follows.

Case 2a. If the next edge after the one from $r_i$ to $s_j$ is from $r_l$ to $s_j$, then $r_l<1$. In this case, we can perturb $r_i$ and $r_l$ by adding or subtracting small positive numbers without changing the direction of the edges. More specifically, we consider
\begin{equation*}
R'=(r_1,\ldots, r_i+\epsilon_1,\ldots, r_l-\epsilon_1,\ldots, r_n),
\end{equation*}
and
\begin{equation*}
R''=(r_1,\ldots, r_i-\epsilon_2,\ldots, r_l+\epsilon_2,\ldots, r_n),
\end{equation*}
where $|R'|=|R''|=|R|$.
To keep the direction of the edges in the original graph $\mathcal{G}(A)$ unchanged, we need to consider two different cases. If there are no edges ending in $r_i$ before the edge from $r_i$ to $s_j$, $\epsilon_1$ must satisfy $\epsilon_1\leq r_l$ and $r_i+\epsilon_1\leq s_l$. In this case $r_i+r_l\leq s_j$ so we get $\epsilon_1\leq r_l$. If there are edges ending in $r_i$ before the edge from $r_i$ to $s_j$, we denote the $i$th row sum just before the construction of the element $a_{i,j}$ by $\bar{r}_i$. In this case $\bar{r}_i+r_l\leq s_j$. Thus $\epsilon_1$ must satisfy $\epsilon_1+r_i\leq 1, \bar{r}_i+\epsilon_1\leq s_j$ and $\epsilon_1\leq r_l$. We get $\epsilon_1\leq\min\{r_l, 1-r_i\}$. Similarly we can get the restriction on $\epsilon_2$. If there are no edges ending in $r_l$ before the edge from $r_l$ to $s_j$, then $\epsilon_2\leq r_i$. If there exist such edges ending in $r_l$ before the edge from $r_l$ to $s_j$, then $\epsilon_2\leq\min\{r_i, 1-r_l\}$.

We claim that there exists a matrix $A_1\in\mathfrak{E}_n(R',S)$ and a matrix $A_2\in\mathfrak{E}_n(R'',S)$, such that $A=\lambda A_1+(1-\lambda) A_2$ where $\lambda=\frac{\epsilon_2}{\epsilon_1+\epsilon_2}$. Actually, $A_1$ is just the matrix that keep all the other elements in $A$ unchanged except for  $a_{i,j}$ and $a_{l,j}$. The element $a_{i, j}$ is replaced by $a_{i, j}+\epsilon_1$, and $a_{l, j}$ is replaced by $a_{l, j}-\epsilon_1$. Similarly, if we replace $a_{i, j}$ and $a_{l, j}$ by $a_{i, j}-\epsilon_2$ and $a_{l, j}+\epsilon_2$ respectively in $A$, then we get $A_2$. This is all because the $\epsilon_1$ and $\epsilon_2$ that we choose do not change the direction of the edges if we apply the same order of the construction of $A$ on $(R',S)$ and $(R'',S)$, respectively.

Case 2b. If there is an edge from $s_j$ to $r_l$, and $r_l<1$, then we can also perturb $r_i$ and $r_l$ by adding or subtracting small positive numbers without changing the direction of the edges. More specifically, we can also get
\begin{equation*}
R'=(r_1,\ldots, r_i+\epsilon_1,\ldots, r_l-\epsilon_1,\ldots, r_n),
\end{equation*}
and
\begin{equation*}
R''=(r_1,\ldots, r_i-\epsilon_2,\ldots, r_l+\epsilon_2,\ldots, r_n),
\end{equation*}
where $|R'|=|R''|=|R|$. To keep the directions of edges unchanged in the construction of $A_1\in\mathfrak{E}_n(R',S)$ and $A_2\in\mathfrak{E}_n(R'',S)$, we need to give some restrictions on $\epsilon_1$ and $\epsilon_2$. Notice that $r_i+\epsilon_1\leq s_j$ and $\epsilon_1\leq r_l$. Since $r_i+r_l\geq s_j$ which implies $s_j-r_i\leq r_l$, we have $\epsilon_1\leq s_j-r_i$. Also $\epsilon_2$ should satisfy that $\epsilon_2\leq r_i$ and $\epsilon_2\leq 1-r_l$. Therefore we have $\epsilon_2\leq\min\{r_i, 1-r_l\}$. Just as in case 2a, there also exist $A_1\in\mathfrak{E}_n(R',S)$ and $A_2\in\mathfrak{E}_n(R'',S)$ such that $A=\lambda A_1+(1-\lambda) A_2$ where $\lambda=\frac{\epsilon_2}{\epsilon_1+\epsilon_2}$.
Here $A_1$ can be obtained from $A$ via replacing $a_{i,j}$ and $a_{l, j}$ by $a_{i ,j}+\epsilon_1$ and $a_{l, j}-\epsilon_1$, respectively. $A_2$ can be obtained from $A$ by replacing $a_{i,j}$ and $a_{l, j}$ by $a_{i ,j}-\epsilon_2$ and $a_{l, j}+\epsilon_2$, respectively.

Case 2c. If there is an edge from $s_j$ to $r_l$ and $r_l=1$, then we need to consider the value of $s_j$. If $s_j=1$, then we can remove column $j$ and row $l$ from $A$ to get a submatrix $A'$. Note that both the row and column sum vectors of $A'$ contains one less entry $1$ compared with $R_{n_0}$ and $S_{n_0}$. By induction the assumption holds in this case.

If $s_j<1$ and there exists another $s_{j'}<1$ connected with $r_l$ by an edge, then by transposing $A$ it can be turned into either case 2a or case 2b as we discussed before. If $s_j<1$ and there exists another $s_{j'}=1$ connected with $r_l$ by an edge, then it can be reduced into considering the submatrix $A'$ with one less entry $1$ in the row and column sum vectors compared with $R_{n_0}$ and $S_{n_0}$. Such an $A'$ can be obtained by deleting column $j'$ and row $l$ from $A$. Again by induction this case can be proved.
\end{proof}
\begin{remark}
From the proof in the above theorem, for $A\in\mathfrak{E}_n(R,S)$ we can always write $A=\lambda A_1+(1-\lambda) A_2$ where $A_1\in\mathfrak{E}_n(R',S)$, $A_2\in\mathfrak{E}_n(R'',S)$ and $|R'|=|R''|=|R|$. There are more $0$ or $1$ in the entries of $R'$ and $R''$ than those in $R$. Performing this process finitely many times, we can eventually express $A$ as a convex combination of extremal matrices in $\mathfrak{E}_n(R_{n_0},S_{n_0})$.
\end{remark}

\begin{remark}
From Theorem~\ref{thm:extrs} and Theorem~\ref{thm:iffext} we can get the equivalence of (a) and (c) in Theorem \ref{main}.
\end{remark}

\section{Extreme points of $\omega_{n,k}$}\label{sec4}
In this section, we characterize the extreme points of $\omega_{n,k}$ via the extreme points of $\omega_n^s$.
\begin{lemma} Suppose $A\in \mathfrak{E}(\omega_n^s)$, where $n-k < s < n-k+1$. We claim that $A$ can be written as a convex combination of $A_1$ and $A_2$, where $A_1\in \mathfrak{E}(\omega_n^{n-k+1})$ and $A_2\in \mathfrak{E}(\omega_n^{n-k}).$
\end{lemma}\label{lem:ewnk}
\begin{proof}
Due to (c) in Theorem \ref{main},  assume that $A$ is in the form
\begin{equation*}
I_{\lceil s\rceil-m}\oplus B_m(s+1-\lceil s\rceil)\oplus O_{n-\lceil s\rceil},
\end{equation*}
For some $0\leq m\leq \lceil s\rceil$.
Let
\begin{equation*}
A_1=I_{\lceil s\rceil-m}\oplus B_m(1)\oplus O_{n-\lceil s\rceil},
\end{equation*}
and
\begin{equation*}
A_2=I_{\lceil s\rceil-m}\oplus B_m(0)\oplus O_{n-\lceil s\rceil}.
\end{equation*}
Notice that $A_1\in\mathfrak{E}(\omega_n^{n-k+1})$ and $A_2\in \mathfrak{E}(\omega_n^{n-k})$. Let $\lambda=s+1-\lceil s\rceil$, and we get
\begin{equation*}
A=\lambda A_1+(1-\lambda)A_2.
\end{equation*}
\end{proof}
The following theorem is a direct consequence of Lemma~\ref{lem:ewnk}.
\begin{theorem} \label{leilm5.1}
\begin{equation*}
\mathfrak{E}(\omega_{n,k})=\mathfrak{E}(\omega_n^{n-k})\cup \mathfrak{E}(\omega_n^{n-k+1}).
\end{equation*}
\end{theorem}

\begin{remark}
Note that $\mathfrak{E}(\omega_n^{n-k+1})\subset\mathfrak{E}(\omega_{n,k}),$ but $\omega_n^{n-k+1}\cap\omega_{n,k}=\emptyset$. Indeed, matrices in $\mathfrak{E}(\omega_n^{n-k+1})$ are limit points of $\omega_{n,k}.$
\end{remark}

\section{Some Applications}\label{sec5}
Through these three partitions, some results of doubly stochastic matrices have been extended to doubly substochastic matrices.  Let $A,B\in \Omega_n.$ Denote by $h$ the maximum diagonal function. A consequence of the main result in \cite{Bala1977} is $$1\leq h(A)+h(B)-h(AB)\leq n.$$ Similar inequality can be generalized to the doubly substochastic matrices $C,D \in \omega_{n,k}$ case.
\begin{theorem}[Theorem 3.3, \cite{CCDKL2017}]
Let $A,B\in \omega_{n,k}.$ Then  $$ \frac{n-k}{n} \leq h(A)+h(B)-h(AB)\leq \min\{n, 2(n-k+1)\}.$$
In particular when $k\geq \frac{n}{2}+1,$
$$\sup_{A,B\in \omega_{n,k}} \{h(A)+h(B)-h(AB)\}=2(n-k+1).$$
\end{theorem}
Denote by $per(A)$ the permanent of $A.$ The following results were given in \cite{CCKL2017} and \cite{CC2017} respectively.
\begin{lemma}[Lemma 2.2, \cite{CCKL2017}]
If $A\in\omega_{n,k}$, then
\begin{equation*}
0\leq per(A)\leq \{\frac{\sigma(A)}{n}\}^n<\left(\frac{n-k+1}{n}\right)^n.
\end{equation*}
\end{lemma}

\begin{theorem}[Theorem 3.1, \cite{CC2017}]\label{main1}
Let $A\in \omega_n$, satisfying either
\begin{enumerate}
\item $n$ is even, or
\item $n$ is odd and $\sigma(A)\leq n-1$.
\end{enumerate}
Let $\sigma(A)=s$ and denote by $t$ the greatest even integer less than or equal to $s$. Then
\begin{equation*}
\max \{{\rm per}(I-A)\ |\ A\in\omega_n^s \}= 2^{t/2}\left[1+\left(\frac{s-t}{2}\right)^2\right].
\end{equation*}
\end{theorem}
Theorem~\ref{main1} can also be rephrased with respect to the sub-defect $k$ as the following corollary.
\begin{corollary} [Corollary 3.2, \cite{CC2017}]
Let $A\in \omega_{n,k}$, where either
\begin{enumerate}
\item $n$ is even, or
\item $n$ is odd and $k>1$.
\end{enumerate}
Denote by $t$ the greatest even integer less than or equal to $n-k+1$. Then
\begin{equation*}
\sup \{{\rm per}(I-A)\ |\ A\in\omega_{n,k}\}= 2^{t/2}\left[1+\left(\frac{n-k+1-t}{2}\right)^2\right].
\end{equation*}
\end{corollary}

Theorem \ref{main1} is a refinement of the main result in \cite{MM1986}.

%

%
%
%
%
%
%
%
%

\bibliographystyle{plain}

\begin{thebibliography}{1}

\bibitem{Bala1977}
K.~Balasubramanian.
\newblock Diagonal sums of doubly stochastic matrices.
\newblock {\em Sankhy\=a Ser. B}, 39(1):89--91, 1977.


\bibitem{RAB1968}
Richard Brualdi.
\newblock Convex sets of non-negative matrices.
\newblock {\em Canadian Journal of Mathematics}, 120: 144-157, 1968.

\bibitem{MM1986}
Massoud Malek.
\newblock On the maximum of per({I}-{A}).
\newblock {\em Linear and Multilinear Algebra}, 19(4):347--355, 1986.



\bibitem{GB1946}
Garrett Birkhoff.
\newblock Tres observaciones sobre el algebra lineal.
\newblock {\em Univ. Nac. Tucum\'{a}n, Revista, Ser. A}, 5: 147--151, 1946.

\bibitem{LeiGLA}
Lei Cao.
\newblock A Short Note on Doubly Substochastic Analogue of Birkhoff's Theorem.
\newblock{\em To appear on Electron. J. Linear Al.}

\bibitem{CK2017}
Lei Cao, Selcuk Koyuncu.
\newblock Sub-defect of product of doubly substochastic matrices.
\newblock {\em Linear Multilinear A.}, 65(4): 653--657, 2017.


\bibitem{CKP2016}
Lei Cao, Selcuk Koyuncu, Timmothy Parmer.
\newblock A mininal completion of doubly substochastic matrix.
\newblock {\em Linear Multilinear A.}, 64(11): 2313--2334, 2016.

\bibitem{CC2017}
Lei Cao, Zhi Chen.
\newblock On the maximum of permanent of (I-A).
\newblock {\em submitted}.

\bibitem{CCDKL2017}
Lei Cao, Zhi Chen, Xuefeng Duan, Selcuk Koyuncu, Huilan Li.
\newblock Diagonal Sums of Doubly Substochastic Matrices.
\newblock {\em submitted}.

\bibitem{CCKL2017}
Lei Cao, Zhi Chen, Selcuk Koyuncu, Huilan Li.
\newblock Permanents of Doubly Substochastic Matrices.
\newblock {\em submitted}.

\bibitem{GPE1980}
Georgy P. Egorychev.
\newblock A solution of the Van der Waerden's permanent problem.
\newblock {\em Preprint IFSO-L3 M, Academy of Sciences SSSR, Krasnoyarks}, 1980.

\bibitem{DIF1981}
Dmitry Falikman.
\newblock proof of the van der Waerden conjecture regarding to the permanent of a doubly stochastic matrix.
\newblock {\em Math. Notes}, 29(6): 475--479, 1981.

\bibitem{PMG1966}
Peter M.Gibson.
\newblock A short proof of an inequality for the permanent function.
\newblock {\em Proc. Amer. Math. Soc.}, 17(2): 535--536, 1966.

\bibitem{PMG1968}
Peter M.Gibson.
\newblock An inequality between the permanent and deternimant.
\newblock {\em Proc. Amer. Math. Soc.}, 19(4): 971--972, 1968.

\bibitem{MK1970}
M. Katz.
\newblock On the extreme points of a certain convex polytope.
\newblock {\em J. Combinatorial Theory}, 8(4): 417--423, 1970.

\bibitem{MK1972}
M. Katz.
\newblock On the extreme points of the set of substochastic and symmetric matrices.
\newblock {\em J. Combinatorial Theory}, 37(3): 576--579, 1972.

\bibitem{MMHM1962}
Marvin Marcus, Henryk Minc.
\newblock Some results on doubly stochastic matrices.
\newblock {\em Proc. Amer. Math. Soc.}, 13(4): 571--579, 1962.

\bibitem{MMHM1964}
Marvin Marcus, Henryk Minc.
\newblock Inequalities for general matric functions.
\newblock {\em Bull. Amer. Math. Soc.}, 70(70): 308--313, 1964.

\bibitem{MN1962}
Marvin Marcus,  Morris Newman.
\newblock Inequalities for the permament function.
\newblock {\em Ann. of Math.}, 75(2): 47--62, 1962.

\bibitem{HM1987}
Henryk Minc.
\newblock Theory of permanents 1982-1985.
\newblock {\em Linear Multilinear A.}, 21(2): 109--148, 1987.

\bibitem{MI1959}
L. Mirsky.
\newblock On a convex set of matrices.
\newblock {\em Arch. Math.}, 10(1): 88--92, 1959.

\bibitem{Wang1974}
Edward Tzu-Hsia Wang.
\newblock Maximum and minimum diagonal sums of doubly stochastic matrices.
\newblock {\em Linear Algebra Appl.}, 8(6): 483--505, 1974.

\bibitem{VDW}
B. L. van der Waerden.
\newblock Aufgabe 45.
\newblock {\em Jber. Deut. Math.-Verein.}, 35: 117, 1926.


\bibitem{Gen2}
D. I. Falikman.
\newblock Proof of the van der Waerden conjecture regarding the permanent of a doubly stochastic matrix.
\newblock {\em Mathematical notes of the Academy of Sciences of the USSR}, 29(6): 475--479, 1981.


\bibitem{Gen1}
Richard Sinkhorn.
\newblock A Relationship Between Arbitrary Positive Matrices and Doubly Stochastic Matrices.
\newblock {\em The Annals of Mathematical Statistics}, 2(35): 876--879, 1964.


\bibitem{TR2}
Richard A. Brualdi
\newblock Some applications of doubly stochastic matrices.
\newblock {\em Linear Algebra and its Applications}, 107(Supplement C): 77--100, 1988.


\bibitem{TR1}
W.B Jurkat, H.J Ryser.
\newblock Term ranks and permanents of nonnegative matrices.
\newblock {\em Journal of Algebra}, 5(3): 342 - 357, 1967.

\bibitem{CP2001}
Soo-Jin Cho, Yun-Sun Nam.
\newblock Convex polytopes of generalized doubly stochastic matrices.
\newblock {\em Communications of the Korean Mathematical Society}, 4(16): 679-690, 2001


\bibitem{EB1972}
Ethan D Bolker.
\newblock Transportation polytopes.
\newblock {\em In Journal of Combinatorial Theory, Series B}, 13(3): 251--262, 1972.

\bibitem{KW1968}
V. Klee, C. Witzgall.
\newblock Facets and vertices of transportation polytope.
\newblock {\em Amer. Math. Soc.}, Providence, 277¨C-282, 1968.

\bibitem{VNJ}
John von Neumann.
\newblock A certain zero-sum two-person game equivalent to the optimal assignment problem.
\newblock {\em Contributions to the theory of games}, 2(28): 5--12, 1953.

\bibitem{CR}
Charles R. Johnson.
\newblock Row Stochastic Matrices Similar to Doubly Stochastic Matrices.
\newblock{\em Linear and Multilinear Algebra}, 10(2), 113--130, 1981.

\bibitem{sinkhorn1967}
Richard Sinkhorn, Paul Knopp.
\newblock Concerning nonnegative matrices and doubly stochastic matrices.
\newblock {\em Pacific J. Math.}, 2(21): 343--348, 1967.

\bibitem{MO}
A. Marshall, I. Olkin.
\newblock Inequalities: theory of majorization and its applications.
\newblock {\em Academic Press}, 1979.

\bibitem{GB1}
P. M. Gibson.
\newblock An ineyuality between the permanent and determinant.
\newblock {\em Proceedings of the American Mathematical Society}, 4(19): 971 -- 972, 1968.

\bibitem{BRUA1977}
Richard A. Brualdi, Peter M. Gibson.
\newblock Convex polyhedra of doubly stochastic matrices. I. Applications of the permanent function.
\newblock {\em Journal of Combinatorial Theory, Series A}, 2(22): 194 -- 230, 1977.

\bibitem{GB2}
P.M. Gibson.
\newblock A short proof of an inequality for the permanent function.
\newblock {\em Proceedings of the American Mathematical Society}, 2(17): 535 -- 536, 1966.

\bibitem{BN1966}
Richard A. Brualdi, Morris Newman.
\newblock Proof of a permanental inequality.
\newblock {\em Quart. J. Math. Oxford Ser.}, 2(17): 234--238, 1966.

\bibitem{Eigen1}
Miroslav Fiedler.
\newblock Bounds for eigenvalues of doubly stochastic matrices.
\newblock {\em Linear Algebra and its Applications}, 3(5): 299 -- 310, 1972.



\bibitem{ANDO}
T. Ando.
\newblock Majorization, doubly stochastic matrices, and comparison of eigenvalues.
\newblock {\em Linear Algebra and its Applications}, Supplement C(118): 163 -- 248, 1989.



\end{thebibliography}

\end{document}